\documentclass[journal]{article}

\usepackage{tikz}
\usepackage{tkz-fct}
\usepackage{pgfplots}
\usepackage{graphicx,url}
\usepackage{dcolumn}
\usepackage{bm}
\usepackage{amsmath,amsfonts,amssymb}
\usepackage[british]{babel}

\newtheorem{definition}{\noindent{\it Definition}}[section]
\newtheorem{theorem}{\noindent{\it Theorem}}[section]

\newtheorem{remark}[theorem]{\noindent{\it Remark}}
\newtheorem{corollary}[theorem]{\noindent{\it Corollary}}
\newenvironment{proof}{\noindent{\it Proof:}}{$\hfill$ $\Box$\\ }
\newtheorem{example}{\noindent{\it Example}}[section]

\begin{document}

\title{Lyapunov exponent for Lipschitz maps}

\author{Giuliano G. La Guardia and Pedro Jeferson Miranda
\thanks{Giuliano Gadioli La Guardia (corresponding author) is with Department of Mathematics and Statistics,
State University of Ponta Grossa (UEPG), 84030-900, Ponta Grossa,
PR, Brazil, e-mail:gguardia@uepg.br. Pedro Jeferson Miranda is with
Department of Physics, State University of Ponta Grossa (UEPG),
84030-900, Ponta Grossa, PR, Brazil. }}

\maketitle

\begin{abstract}
It is well-known that the Lyapunov exponent plays a fundamental role
in dynamical systems. In this note, we propose a definition of
Lyapunov exponent for Lipschitz maps, which are not necessarily
differentiable. Additionally, we show that the main results which
are valid to discrete standard dynamical systems are also true when
considering Lipschitz maps instead of considering differentiable
maps. Therefore, this novel approach expands the theory of dynamical
systems.
\end{abstract}

\section{Introduction}

Theory of dynamical systems is extensively investigated in the
literature
\cite{Lyapunov:1992,Ledrappier:1981,Young:1982,Grebogi:1990,Pattanayak:1999,Aniszewska:2008,Dabrowski:2012,Sadri:2014,Hu:2017}.
Lyapunov, in his fabulous work \cite{Lyapunov:1992}, made several
important contributions in the investigation of the stability of
motion. In fact, the Lyapunov exponent strongly characterizes the
behavior of the system. In fact, all the papers mentioned above have
dealt with investigations and computations of the Lyapunov exponent
of the corresponding systems in order to characterize them.

The main contributions of this note are to propose a definition of
Lyapunov exponent for Lipschitz maps as well as to show that the
results which are valid to discrete standard dynamical systems also
hold when considering Lipschitz maps instead of considering
differentiable maps. Moreover, since a Lipschitz map is not
necessarily differentiable (recall that a Lipschitz map $f:{\mathbb
R}\longrightarrow{\mathbb R}$ satisfies that following condition:
for all $x , y \in {\mathbb R}$ with $x\neq y$, one has $\frac{|
f(x) - f(y) |}{| x - y |} \leq c $, for some $c \in {\mathbb R}, c >
0$; the existence of the limit is not guaranteed), this novel
approach expands the theory of discrete dynamical systems.

Another advantage of considering Lipschitz maps instead of
differentiable ones is that it is not necessary to compute the
derivative of the map in order to find fixed points which are sinks
or sources. In fact, only by the feature of the map (Lipschitz or
reverse Lipschitz) one can find which fixed points are sinks or
sources. Thus, our method is easier to be applied when compared to
the standard method (which only deals with differentiable maps).

Generalizations of Lyapunov exponent defined over continuous maps
were presented in the literature \cite{Kunze:2000,Barreira:2005}. In
\cite{Kunze:2000}, the author considered Lyapunov exponent for
non-smooth systems in order to apply to a pendulum with dry
friction. In \cite{Barreira:2005}, the authors defined Lyapunov
exponent for continuous map (see Subsection 3.1). However, such
definition is a little complex to be applied in practice. Otherwise,
our new approach is simple to be applied and it is complete in the
sense that it characterizes the concept of sources, sinks, Lyapunov
exponent, Lyapunov number and all the standard concepts and results
of discrete dynamical systems in a natural way. More precisely, the
results which hold for differentiable maps also hold in our new
context, that is, also hold for Lipschitz maps.

This note is organized as follows. Section~\ref{sec2} presents the
main concepts that will be utilized in this work. In
Section~\ref{sec3}, we present the contributions of the paper. More
precisely, we propose a definition of Lyapunov exponent for
Lipschitz maps which are not necessarily differentiable.
Additionally, we show that the main results which are valid to
discrete standard dynamical systems are also true when considering
Lipschitz maps instead of considering differentiable maps. Finally,
in Section~\ref{sec4}, a brief summary of this work is drawn.

\section{Preliminaries}\label{sec2}

In this section, we present the known results and concepts for the
development of this work. Throughout this note, we denote by
${\mathbb R}$ the field of real numbers and ${\mathbb R}^{m}$ is the
$m$-dimensional vector space over ${\mathbb R}$. We only consider
discrete dynamical systems.

As usual, a function whose domain is equal to its range is called
\emph{map}. Let $f:A\longrightarrow A$ be a map and $x \in A$. The
\emph{orbit} ${\mathcal O}_{x}$ of $x$ under $f$ is the set of
points ${\mathcal O}_{x} = \{x, f(x), f^{2}(x), \ldots, \}$, where
$f^{2}(x)=f(f(x))$ and so on. The point $x$ is said to be the
\emph{initial value} of the orbit. If there exists a point $p$ in
the domain of $f$ such that $f(p)=p$ then $p$ is called a
\emph{fixed point} of $f$.

Let $f:{\mathbb R}\longrightarrow{\mathbb R}$ be a map. Recall that
$f$ is said to be \emph{Lipschitz} if there exists a constant $c \in
{\mathbb R}$, $c > 0$ (called Lipschitz constant of $f$), such that
$\forall \ x, y \in {\mathbb R} \Longrightarrow | f(x) - f(y) | \leq
c | x - y |$, where $| \cdot |$ denotes the absolute value function
on ${\mathbb R}$. In other words, if $x\neq y$ then $\frac{| f(x) -
f(y) |}{| x - y |} \leq c $, i.e., the quotient is bounded. If
$\forall x, y \in {\mathbb R} \Longrightarrow | f(x) - f(y) | < c |
x - y |$, then $f$ is called \emph{strictly Lipschitz}.

Given $x \in {\mathbb R}$, the \emph{epsilon neighborhood}
$N_{\epsilon} (x)$ of $x$ is defined as $N_{\epsilon} (x)=\{y \in
{\mathbb R} : |x - y| < \epsilon \}$. Let $f:{\mathbb
R}\longrightarrow{\mathbb R}$ be a map and $x\in {\mathbb R}$. We
say that $f$ is \emph{locally Lipschitz at $x$} if there exists an
$\epsilon$-neighborhood $N_{\epsilon} (x)$ of $x$ such that $f$
restricted to $N_{\epsilon} (x)$ is Lipschitz.

Here, we introduce in the literature the new concept of
\emph{reverse Lipschitz map}. This concept will be utilized in order
to characterize sources (see Definition~\ref{defss} in the
following).

\begin{definition}\label{def2a}
Let $f:{\mathbb R}\longrightarrow{\mathbb R}$ be a map. We say that
$f$ is reverse Lipschitz (RL) if there exists a constant $c \in
{\mathbb R}$, $c > 0$ (called reverse Lipschitz constant of $f$)
such that, $\forall \ x, y \in {\mathbb R} \Longrightarrow | f(x) -
f(y) | \geq c | x - y |$. Similarly, $f$ is called locally reverse
Lipschitz at $x$ if there exists an $\epsilon$-neighborhood
$N_{\epsilon} (x)$ of $x$ such that $f$ restricted to $N_{\epsilon}
(x)$ is reverse Lipschitz.
\end{definition}

\begin{definition}\label{defss}
Let $f:{\mathbb R}\longrightarrow{\mathbb R}$ be a map and $p$ be a
fixed point of $f$. One says that $p$ is a \emph{sink} (or
\emph{attracting fixed point}) if there exists an $\epsilon > 0$
such that, for all $x \in N_{\epsilon}(p)$,
$\displaystyle{\lim_{k\rightarrow\infty}}f^{k}(x) = p$. On the other
hand, if all points sufficiently close to $p$ are repelled from $p$,
then $p$ is called a \emph{source}. In other words, $p$ is a source
if there exists an epsilon neighborhood $N_{\epsilon} (p)$ such
that, for every $\ x \in N_{\epsilon} (p)$, $x \neq p$, there exists
a positive integer $k$ with $|f^{k}(x) - p| \geq \epsilon$.
\end{definition} 

\section{The Results}\label{sec3}

In this section, we present the contributions of this paper. We
divide the section into four subsections: the stability of fixed
points in ${\mathbb R}$, stability of periodic orbits, stability of
maps on the Euclidean space ${\mathbb R}^{n}$, and a new definition
of Lyapunov exponent for Lipschitz maps.

\subsection{Stability of fixed points in ${\mathbb R}$}\label{sec3.1}

We begin this subsection by recalling a well-known result shown in
the literature.

\begin{theorem}\cite{Alligood:1996,Martelli:1999}\label{thm1}
Let $f:{\mathbb R}\longrightarrow{\mathbb R}$ be a
smooth map and let $p$ is a fixed point of $f$. If $|f^{'} (p)| <
1$, then $p$ is a sink. On the other hand, if $|f^{'} (p)| > 1$,
then $p$ is a source.
\end{theorem}

From here to the end of the paper, we show that exchanging the
differentiability condition to the Lipschitz condition, the results
which hold for standard discrete dynamical systems are also true in
this new context. Since a Lipschitz map do not need to be
differentiable (remember that for $x\neq y$, it follows that
$\frac{| f(x) - f(y) |}{| x - y |} \leq c $; the existence of the
limit is not guaranteed), we can utilize this new approach for a
wider class of maps.

Theorem~\ref{main1}, shown in the following, is the first
contribution of this paper.

\begin{theorem}(Stability test for fixed points)\label{main1}
Let $f:{\mathbb R}\longrightarrow{\mathbb R}$ be a map and $p \in
{\mathbb R}$ a fixed point of $f$.
\begin{itemize}
\item [ 1-] If $f$ is strictly
locally Lipschitz map at $p$, with Lipschitz constant $c < 1$, then
$p$ is a sink.

\item [ 2-] If $f$ is locally reverse Lipschitz map at $p$, with constant $r > 1$, then
$p$ is a source.
\end{itemize}
\end{theorem}
\begin{proof}
To show Item 1-), let $f$ be a strictly locally Lipschitz map at $p$
with Lipschitz constant $c < 1$. Then there exists an
$\epsilon$-neighborhood $N_{\epsilon}(p)$ of $p$ such that $|f(x) -
f(p)| < c|x - p|$ for all $x \in N_{\epsilon}(p)$. Therefore, if $x
\in N_{\epsilon}(p)$ then $|f(x) - f(p)|= |f(x) - p| < c |x - p| <
|x - p| < \epsilon$, i.e., $f(x) \in N_{\epsilon}(p)$. Applying the
same argument, it follows that $f^{2}(x), f^{3}(x), \ldots ,
f^{n}(x), \ldots$ also belong to $N_{\epsilon}(p)$. Next, we will
prove by induction that the inequality $|f^{k}(x) - p|< c^{k}|x-p|$,
$\forall \ x \in N_{\epsilon}(p)$, holds for all $k\geq 1$. It is
clear that for $k=1$ the inequality holds. Assume that the
inequality is true for $k$: $|f^{k}(x) - p|< c^{k}|x - p|$. We must
prove that $|f^{k+1}(x) - p|< c^{k+1}|x-p|$ is also true. As $f$ is
strictly locally Lipschitz in $N_{\epsilon}(p)$ and since $f^{k}(x)
\in N_{\epsilon}(p)$ we know that $|f^{k+1}(x) - p|< c|f^{k}(x) -
p|$. From induction hypothesis one has $|f^{k+1}(x) - p| < c^{k+1}|x
- p|$ and the result follows. Since $c < 1$ it follows that
$\displaystyle{\lim_{k\rightarrow\infty}}c^{k+1}|x - p|=0$. Thus
$\displaystyle{\lim_{k\rightarrow\infty}}f^{k}(x) = p$, i.e., $p$ is
a sink, as desired.

In order to prove Item 2-). From hypothesis, we know that there
exists an $\epsilon$-neighborhood $N_{\epsilon}(p)$ of $p$ such that
$|f(x) - p| \geq r|x - p|$ for all $x \in N_{\epsilon}(p)$. Fix $x
\in N_{\epsilon}(p)$, $x \neq p$. If $|f(x) - p| \geq \epsilon$,
then the result follows. Otherwise, $|f(x) - p| < \epsilon$, which
implies that $f(x) \in N_{\epsilon}(p)$. Applying again the fact
that $f$ is locally reverse Lipschitz in $N_{\epsilon}(p)$, one has
$|f^{2}(x) - p| \geq r|f(x) - p|\geq r^{2}|x - p|$. If $|f^{2}(x) -
p|\geq \epsilon$, the result holds. Otherwise, $|f^{2}(x) - p| <
\epsilon$, i.e., $f^{2}(x)\in N_{\epsilon}(p)$. Because $f^{2}(x)\in
N_{\epsilon}(p)$ and since $f$ is locally reverse Lipschitz in
$N_{\epsilon}(p)$, it follows that $|f^{3}(x) - p| \geq r|f^{2}(x) -
p|\geq r^{3}|x - p|$. If $|f^{3}(x) - p|\geq \epsilon$, then the
result is true. Otherwise, we proceed similarly as above. Applying
repeatedly this reasoning, it follows that there exists an integer
$k^{*} \geq 1$ such that $|f^{k^{*}}(x) - p| \geq r^{k^{*}}|x -
p|\geq \epsilon$, i.e., $|f^{k^{*}}(x) - p| \geq \epsilon$. More
precisely, because $r > 1$ and since $|x - p|$ is a fixed positive
real number, there exists a sufficiently large positive integer
$k^{*}$ such that $r^{k^{*}}|x - p|\geq \epsilon$, which implies
that $|f^{k^{*}}(x) - p| \geq \epsilon$ holds. Therefore, $p$ is a
source. The proof is complete.
\end{proof}

\begin{corollary}\label{main1.1}
Let $f:{\mathbb R}\longrightarrow{\mathbb R}$ be a map.
\begin{itemize}
\item [ 1-] If $f$ is strictly Lipschitz, with constant $c < 1$,
then there exists only one fixed point $p$ which is a sink.

\item [ 2-] If $f$ is reverse Lipschitz with constant $r > 1$,
then all fixed point $p$ is a source.
\end{itemize}
\end{corollary}
\begin{proof}
Item 1-) is the well-known Banach contraction theorem on the real
line (see \cite[Thm. 5.2.1]{Martelli:1999}).

To show Item 2-), it suffices to utilize an analogous proof of
Theorem~\ref{main1} in the following way. From hypothesis, we know
that $|f(x) - p| \geq r|x - p|$ for all $x \in {\mathbb R}$. We fix
$x \in {\mathbb R}$, $x \neq p$. If $|f(x) - p| \geq \epsilon$, then
the result follows. Otherwise, applying the fact that $f$ is reverse
Lipschitz at $p$, one has $|f^{2}(x) - p| \geq r|f(x) - p|\geq
r^{2}|x - p|$. If $|f^{2}(x) - p|\geq \epsilon$, the result holds.
Otherwise, $|f^{2}(x) - p| < \epsilon$, Since $f$ is is reverse
Lipschitz at $p$, it follows that $|f^{3}(x) - p| \geq r|f^{2}(x) -
p|\geq r^{3}|x - p|$. If $|f^{3}(x) - p|\geq \epsilon$, then one has
the result. Otherwise, applying repeatedly this reasoning, there
will be an integer $k_0 \geq 1$ such that $|f^{k_0}(x) - p| \geq
\epsilon$.
\end{proof}

\subsection{Stability of periodic points in ${\mathbb R}$}\label{sec3.2}

In this subsection, we deal with the stability of periodic points.

We first recall some known results concerning such topic. For more
details, we refer the reader to \cite{Alligood:1996,Martelli:1999}.
Let $f:{\mathbb R}\longrightarrow{\mathbb R}$ be a map and $p \in
{\mathbb R}$. Recall that $p$ is said to be a \emph{periodic point}
of period $k$ (or $k$-periodic point) if $f^{k}(p)=p$ and if $k$ is
the smallest such positive integer. The orbit of $p$ (which consists
of $k$ points) is called a \emph{periodic orbit} of period $k$ (or
$k$-periodic orbit). We will denote the $k$-periodic orbit of $p$ by
${\mathcal O}_{p}^{k}$.

If $f:{\mathbb R}\longrightarrow{\mathbb R}$ is a map and $p$ is a
$k$-periodic point, then the orbit ${\mathcal O}_{p}^{k}$ of $p$ is
called a \emph{periodic sink} if $p$ is a sink for the map $f^{k}$.
Analogously, ${\mathcal O}_{p}^{k}$ is a \emph{periodic source} if
$p$ is a source for $f^{k}$.

Let us recall the stability criteria for periodic orbits.

\begin{theorem}\cite{Alligood:1996,Martelli:1999}\label{thm2}
Let $f:{\mathbb R}\longrightarrow{\mathbb R}$ be a map. If
$|f^{'}(p_1 )\cdots f^{'}(p_k )| < 1$ then the $k$-periodic orbit
${\mathcal O}_{p}^{k}= \{p_1 , \ldots, p_k \}$ is a sink; if
$|f^{'}(p_1 )\cdots f^{'}(p_k )| > 1$ then ${\mathcal O}_{p}^{k}$ is
a source.
\end{theorem}

The following result is new version of Theorem~\ref{thm2} for
(reverse) Lipschitz maps.

\begin{theorem}(Stability test for periodic orbits)\label{main2}
Let $g=f^{k}:{\mathbb R}\longrightarrow{\mathbb R}$ be a map and $p
\in {\mathbb R}$ a fixed point of $g$.
\begin{itemize}
\item [ 1-] If $g$ is strictly
locally Lipschitz map at $p$, with Lipschitz constant $c < 1$, then
${\mathcal O}_{p}^{k}$ is a periodic sink.

\item [ 2-] If $g$ is locally reverse Lipschitz map at $p$, with constant
$r > 1$, then ${\mathcal O}_{p}^{k}$ is a periodic source.
\end{itemize}
\end{theorem}

\begin{proof}
1-) Assume that $g$ is strictly locally Lipschitz map at $p$. Then
there exists an $\epsilon$-neighborhood $N_{\epsilon}(p)$ in which
$g$ is strictly Lipschitz. Furthermore, we have shown in the proof
of Theorem~\ref{main1} that $|g^{l}(x) - p| < c^{l}|x-p|$ for all $x
\in N_{\epsilon}(p)$. From Theorem~\ref{main1}, $p$ is a sink for
the map $g= f^{k}$, i.e., ${\mathcal O}_{p}^{k}$ is a periodic sink.

2-) The proof is the same to that of Item 2-) of Theorem~\ref{main1}
considering the map $g=f^{k}$.
\end{proof}

%

\subsection{Stability of fixed points in ${\mathbb
R}^{m}$}\label{sec3.3}

As usual, we denote vectors in ${\mathbb R}^{m}$ and maps on
${\mathbb R}^{m}$ by boldface letters. Let us consider the
$m$-dimensional real vector space ${\mathbb R}^{m}$ endowed with a
norm $\parallel\cdot
\parallel$ (in particular, the Euclidean norm). Let $\textbf{p} =
(p_1 , \ldots , p_{m} ),  \textbf{v} = (v_1 , \ldots , v_{m} ) \in
{\mathbb R}^{m}$ be two points (vectors). The
$\epsilon$-neighborhood $N_{\epsilon} (\textbf{p})$ of $\textbf{p}$
is defined by $N_{\epsilon} (\textbf{p})=\{\textbf{v} \in {\mathbb
R}^{m} :
\parallel \textbf{v} - \textbf{p} \parallel < \epsilon  \}$.

Let $\textbf{f}:{\mathbb R}^{m}\longrightarrow{\mathbb R}^{m}$ be a
map and let $\textbf{p} \in {\mathbb R}^{m}$ be a fixed point of
$\textbf{f}$, i.e., $\textbf{f}(\textbf{p})=\textbf{p}$. If there
exists an $\epsilon$-neighborhood $N_{\epsilon} (\textbf{p})$ of
$\textbf{p}$ such that $\forall \ \textbf{v} \in N_{\epsilon}
(\textbf{p})$,
$\displaystyle{\lim_{k\rightarrow\infty}}\textbf{f}^{k}(\textbf{v})=\textbf{p}$,
then $\textbf{p}$ is called a sink (or attracting fixed point). If
there exists an $N_{\epsilon} (\textbf{p})$ such that, $\forall \
\textbf{v} \in N_{\epsilon} (\textbf{p})$, except for $\textbf{p}$
itself, eventually maps outside of $N_{\epsilon} (\textbf{p})$, then
$\textbf{p}$ is called a source (or repeller).

If $\textbf{f}$ is a smooth map and $\textbf{p} \in {\mathbb
R}^{m}$, we represent $\textbf{f}$ in terms of its coordinates
functions $\textbf{f} = (f_1 , \ldots , f_{m})$. Let
$\textbf{Df}(\textbf{p})$ be the Jacobian matrix of $\textbf{f}$ at
$\textbf{p}$. With this notation in mind, we can state the following
well-known result:

\begin{theorem}[Thm. 2.11]\cite{Alligood:1996}\label{thm3}
Let $\textbf{f}:{\mathbb R}^{m}\longrightarrow{\mathbb R}^{m}$ be a
map and assume that $\textbf{p} \in {\mathbb R}^{m}$ is a fixed
point of $\textbf{f}$. If the magnitude of each eigenvalue of
$\textbf{Df}(\textbf{p})$ is less than $1$, then $\textbf{p}$ is a
sink; if the magnitude of each eigenvalue of
$\textbf{Df}(\textbf{p})$ is greater than $1$, then $\textbf{p}$ is
a source.
\end{theorem}

On the vector space ${\mathbb R}^{m}$, the concept of Lipschitz map
reads as follows.

Let $\textbf{f}:{\mathbb R}^{m}\longrightarrow{\mathbb R}^{m}$ be a
map. We say that $\textbf{f}$ is Lipschitz if there exists a
constant $c \in {\mathbb R}$, $c
> 0$ such that $\forall \
\textbf{v}, \textbf{w} \in {\mathbb R}^{m} \Longrightarrow
\parallel \textbf{f}(\textbf{v}) - \textbf{f}(\textbf{w}) \parallel \leq c \parallel
\textbf{v} - \textbf{w} \parallel$, where $\parallel \cdot
\parallel$ denotes a norm over ${\mathbb R}^{m}$. If $\forall \ \textbf{v}, \textbf{w}
\in {\mathbb R}^{m} \Longrightarrow \parallel \textbf{f}(\textbf{v})
- \textbf{f}(\textbf{w})
\parallel < c \parallel \textbf{v}- \textbf{w} \parallel$, we say that
$\textbf{f}$ is strictly Lipschitz.

The next result is a natural generalization of Theorem~\ref{main1}
to the Euclidean space ${\mathbb R}^{m}$.

\begin{theorem}(Stability test for fixed points on
${\mathbb R}^{m}$)\label{main3} Let $\textbf{f}:{\mathbb
R}^{m}\longrightarrow{\mathbb R}^{m}$ be a map and let $\textbf{p}
\in {\mathbb R}^{m}$ a fixed point of $\textbf{f}$.
\begin{itemize}
\item [ 1-] If $\textbf{f}$ is strictly
locally Lipschitz map at $\textbf{p}$, with Lipschitz constant $c <
1$, then $\textbf{p}$ is a sink.

\item [ 2-] If $\textbf{f}$ is locally reverse Lipschitz map at $\textbf{p}$, with constant
$r > 1$, then $\textbf{p}$ is a source.
\end{itemize}
\end{theorem}
\begin{proof}
The proofs of both items are the same to that of Theorem~\ref{main1}
only changing the absolute value function $||$ on ${\mathbb R}$ by a
norm $\parallel \parallel$ on ${\mathbb R}^{m}$.

Item 1) We know there exists an $\epsilon$-neighborhood
$N_{\epsilon}(\textbf{p})$ of $\textbf{p}$ such that $\parallel
\textbf{f}(\textbf{x}) - \textbf{f}(\textbf{p})\parallel <
c\parallel \textbf{x} - \textbf{p} \parallel$ for all $\textbf{x}
\in N_{\epsilon}(\textbf{p})$. Therefore, if $\textbf{x} \in
N_{\epsilon}(\textbf{p})$ then $\parallel \textbf{f}(\textbf{x}) -
\textbf{f}(\textbf{p})\parallel < \epsilon$, i.e.,
$\textbf{f}(\textbf{x}) \in N_{\epsilon}(\textbf{p})$. By the same
argument, it follows that ${\textbf{f}^{2}}(\textbf{x}), \ldots ,
{\textbf{f}}^{n}(\textbf{x}), \ldots$ also belong to
$N_{\epsilon}(\textbf{p})$. Applying induction, we can show that
$\parallel {\textbf{f}}^{k}(\textbf{x}) - \textbf{p}\parallel<
c^{k}\parallel \textbf{x}-\textbf{p}\parallel$, $\forall \ \textbf{x
}\in N_{\epsilon}(\textbf{p})$, holds for all $k\geq 1$. Since $c <
1$ it follows that
$\displaystyle{\lim_{k\rightarrow\infty}}c^{k+1}\parallel \textbf{x}
- \textbf{p}\parallel=0$, so $\textbf{p}$ is a sink.

Item 2-) There exists an $\epsilon$-neighborhood
$N_{\epsilon}(\textbf{p})$ of $\textbf{p}$ such that $\parallel
\textbf{f}(\textbf{x}) - \textbf{p}\parallel \geq r\parallel
\textbf{x} - \textbf{p}\parallel$ for all $\textbf{x} \in
N_{\epsilon}(\textbf{p})$. Let us consider $\textbf{x} \in
N_{\epsilon}(\textbf{p})$, $\textbf{x} \neq \textbf{p}$. If
$\parallel \textbf{f}(\textbf{x}) - \textbf{p}\parallel \geq
\epsilon$, there is nothing to prove. Otherwise,
$\parallel\textbf{f}(\textbf{x}) - \textbf{p}\parallel < \epsilon$,
which implies that $\textbf{f}(\textbf{x}) \in
N_{\epsilon}(\textbf{p})$. Because $\textbf{f}$ is locally reverse
Lipschitz in $N_{\epsilon}(\textbf{p})$, one has $\parallel
{\textbf{f}}^{2}(\textbf{x}) - \textbf{p}\parallel \geq
r^{2}\parallel \textbf{x} - \textbf{p}\parallel$. If $\parallel
{\textbf{f}}^{2}(\textbf{x}) - \textbf{p}\parallel \geq \epsilon$,
the result holds. Otherwise, $\parallel \textbf{f}^{2}(\textbf{x}) -
\textbf{p}\parallel < \epsilon$, i.e.,
${\textbf{f}}^{2}(\textbf{x})\in N_{\epsilon}(\textbf{p})$.
Proceeding similarly as in the proof of Theorem~\ref{main1}, the
results follows.
\end{proof}

\begin{corollary}\label{main3.1}
Let $\textbf{f}:{\mathbb
R}^{m}\longrightarrow{\mathbb R}^{m}$ be a map.
\begin{itemize}
\item [ 1-] If $\textbf{f}$ is strictly
Lipschitz map with Lipschitz constant $c < 1$, then there is only
one fixed point $\textbf{p}$ which is a sink.

\item [ 2-] If $\textbf{f}$ is reverse Lipschitz map with constant
$r > 1$, then all fixed point $\textbf{p}$ is a source.
\end{itemize}
\end{corollary}
\begin{proof}
Item 1-) is the well-known Banach contraction theorem (see for
example \cite[Thm. 5.2.1]{Martelli:1999}).

The proof of Item 2) is similar to that of Corollary~\ref{main1.1}.
We present it here for completeness. Since $\textbf{f}$ is locally
reverse Lipschitz map at $\textbf{p}$, it follows that $\parallel
\textbf{f}(\textbf{x}) - \textbf{p}\parallel \geq r\parallel
\textbf{x} - \textbf{p}\parallel$ for all $\textbf{x} \in {\mathbb
R}^{m}$. Let us consider a fixed vector $\textbf{x} \in {\mathbb
R}^{m}$, $\textbf{x} \neq \textbf{p}$. If
$\parallel\textbf{f}(\textbf{x}) - \textbf{p}\parallel \geq
\epsilon$, then the result follows. Otherwise, because $\textbf{f}$
is reverse Lipschitz at $\textbf{p}$, it follows that $\parallel
{\textbf{f}}^{2}(\textbf{x}) - \textbf{p} \parallel \geq r
\parallel \textbf{f}(\textbf{x}) - \textbf{p}\parallel\geq r^{2}\parallel \textbf{x} - \textbf{p}\parallel$. If
$\parallel {\textbf{f}}^{2}(\textbf{x}) - \textbf{p}\parallel\geq
\epsilon$, the result holds. On the other hand, $\parallel
{\textbf{f}}^{2}(\textbf{x}) - \textbf{p}\parallel < \epsilon$.
Applying again the fact that $\textbf{f}$ is is reverse Lipschitz at
$\textbf{p}$ one has $\parallel {\textbf{f}}^{3}(\textbf{x}) -
\textbf{p}\parallel \geq r\parallel {\textbf{f}}^{2}(\textbf{x}) -
\textbf{p}\parallel\geq r^{3}\parallel \textbf{x} -
\textbf{p}\parallel$. If $\parallel {\textbf{f}}^{3}(\textbf{x}) -
\textbf{p}\parallel\geq \epsilon$, then the result follows.
Otherwise, applying repeatedly this procedure, there will be an
integer $k_0 \geq 1$ such that $\parallel
{\textbf{f}}^{k_0}(\textbf{x}) - \textbf{p}\parallel \geq \epsilon$.
\end{proof}

\begin{remark}
Note that the procedure utilized in Subsection~\ref{sec3.2} can be
easily adapted to generate analogous results for the stability of
periodic orbits of maps defined over ${\mathbb R}^{m}$. Since both
proofs are similar, we do not present the last one here.
\end{remark}

\subsection{Lyapunov exponent}\label{sec3.4}

In this subsection, we introduce in the literature the Lyapunov
number and the Lyapunov exponent for Lipschitz maps. We only
consider the case of maps defined over ${\mathbb R}$ (or over any
subset of ${\mathbb R}$), since the procedure for maps on ${\mathbb
R}^{n}$ (or over any subset of ${\mathbb R}$) is quite similar.

We denote by ${\mathcal O}_{x_1}= \{ x_1 , x_2, x_3 , \ldots \}$ an
arbitrary orbit with initial point $x_1 \in {\mathbb R}$, where
$x_2= f(x_1), x_3 = f^{2}(x_1), x_4 = f^{3}(x_1), \ldots$. Assume
that $f$ is a smooth map on ${\mathbb R}$ and $x_1 \in {\mathbb R}$.
Recall that the Lyapunov number $L(x_1)$
\cite{Alligood:1996,Martelli:1999} of the orbit ${\mathcal O}_{x_1}=
\{ x_1 , x_2, x_3 , \ldots \}$ is defined as
$L(x_1)=\displaystyle{\lim_{n\rightarrow\infty}}(|f^{'}(x_1)|\cdots
|f^{'}(x_n)|)^{1/n},$ if the limit exists. The Lyapunov exponent
$h(x_1)$ is defined as $h(x_1)=
\displaystyle{\lim_{n\rightarrow\infty}}$
$(1/n)[\ln|f^{'}(x_1)|+\cdots +\ln|f^{'}(x_n)|],$ if the limit
exists.

Let ${\mathcal O}_{x_1}=\{x_1, x_2, \ldots, x_n, \ldots \}$ be an
orbit and let ${\mathcal O}_{y_1}^{k}=\{y_1, y_2, \ldots, y_k\}$ be
a $k$-periodic orbit (see beginning of Subsection~\ref{sec3.2}). We
say that the orbit ${\mathcal O}_{x_1}$ is \emph{asymptotically
periodic} (see for instance \cite[Definition 3.3]{Alligood:1996}) if
it converges to a periodic orbit ${\mathcal O}_{y_1}^{k}$ for some
integer $k\geq 1$ and $y_1 \in {\mathbb R}$, when
$n\longrightarrow\infty$. In other words, there exists a periodic
orbit $ \{y_1, y_2, \ldots, y_k\}= \{y_1, y_2, \ldots, y_k, y_1,
y_2, \ldots, y_k, \ldots\}$ such that
$\displaystyle{\lim_{n\rightarrow\infty}} |x_n - y_n |=0$.

Until now in this subsection, we recall the concepts of Lyapunov
number and Lyapunov exponent for smooth maps, which are well-known
in the literature. Now, we will propose the definition of Lyapunov
number and Lyapunov exponent for Lipschitz maps, which are not
necessarily differentiable. The first result that will be utilized
to this goal is due to Rademacher:

\begin{theorem}(Rademacher's theorem)[Thm. 3.1.6.]\cite{Federer:1996} (see also
\cite{Rademacher:1919})\label{rad1} If $\textbf{f}: {\mathbb
R}^{m}\longrightarrow {\mathbb R}^{n}$ is a Lipschitz map, then
$\textbf{f}$ is differentiable at Lebesgue almost all points of
${\mathbb R}^{m}$.
\end{theorem}

A variant of this result is given below.

\begin{theorem}\cite[Thm.
3.1]{Heinonen:2004}\label{rad2} Let $\Omega \subset {\mathbb R}^{m}$
be an open set, and let $\textbf{f}: \Omega\longrightarrow {\mathbb
R}^{n}$ be a Lipschitz map. Then $\textbf{f}$ is differentiable at
almost every point (Lebesgue) in $\Omega$.
\end{theorem}

From Rademacher's theorem, we can guarantee that a Lipschitz map
$f:{\mathbb R}\longrightarrow{\mathbb R}$ is differentiable in a set
$X = {\mathbb R}- Y$, where the set $Y$ has zero Lebesgue measure.
In this new context, we can define the Lyapunov number and the
Lyapunov exponent as well as the concept of asymptotically periodic
orbit for Lipschitz maps.

\begin{definition}\label{L1}
Let $f:{\mathbb R}\longrightarrow{\mathbb R}$ be a Lipschitz map and
assume that ${\mathcal O}_{x_1}\subset X$. Then the Lyapunov number
$L(x_1)$ of the orbit ${\mathcal O}_{x_1}= \{ x_1 , x_2, x_3 ,
\ldots \}$ is defined as
\begin{eqnarray}
L(x_1)=\displaystyle{\lim_{n\rightarrow\infty}}(|f^{'}(x_1)|\cdots
|f^{'}(x_n)|)^{1/n},
\end{eqnarray}
if the limit exists.

The Lyapunov exponent $h(x_1)$ is defined as
\begin{eqnarray}
h(x_1)=
\displaystyle{\lim_{n\rightarrow\infty}}(1/n)[\ln|f^{'}(x_1)|+\cdots
+\ln|f^{'}(x_n)|],
\end{eqnarray}
if the limit exists.
\end{definition}

Recall that a map $f: {\mathbb R}\longrightarrow {\mathbb R}$ is
said to be \emph{locally Lipschitz} on an open interval $(a,
b)\subset {\mathbb R}$ if $f$ restricted to $(a, b)$ is Lipschitz.
In terms of locally Lipschitz maps, we have the following variant to
Definition~\ref{L1}.

\begin{definition}\label{L2}
Let $f:{\mathbb R}\longrightarrow{\mathbb R}$ be a locally Lipschitz
map in a (nondegenerate) open interval $(a, b)$, and assume that
${\mathcal O}_{x_1}\subset (a, b)\cap X$. Then the Lyapunov number
$L(x_1)$ of the orbit ${\mathcal O}_{x_1}= \{ x_1 , x_2, x_3 ,
\ldots \}$ is defined as
\begin{eqnarray}
L(x_1)=\displaystyle{\lim_{n\rightarrow\infty}}(|f^{'}(x_1)|\cdots
|f^{'}(x_n)|)^{1/n},
\end{eqnarray}
if the limit exists.

The Lyapunov exponent $h(x_1)$ is defined as
\begin{eqnarray}
h(x_1)=
\displaystyle{\lim_{n\rightarrow\infty}}(1/n)[\ln|f^{'}(x_1)|+\cdots
+\ln|f^{'}(x_n)|],
\end{eqnarray}
if the limit exists.
\end{definition}

Let us reformulate the concept of \emph{asymptotically periodic
orbit} in terms of Lipschitz maps.

\begin{definition}\label{L3}
Let $f:{\mathbb R}\longrightarrow{\mathbb R}$ be a Lipschitz map. An
orbit $\{x_1, x_2, \ldots, x_n, \ldots \}$ is called asymptotically
periodic if it converges to a periodic orbit when
$n\longrightarrow\infty$. In other words, there exists a periodic
orbit $\{y_1, y_2, \ldots, y_k, y_1,$ $y_2, \ldots, y_k, \ldots\}$
such that $\displaystyle{\lim_{n\rightarrow\infty}} |x_n - y_n |=0$.
\end{definition}

The following result is a variant of \cite[Theorem
3.4]{Alligood:1996} based on Lipschitz maps.

\begin{theorem}\label{L4}
Let $f: {\mathbb R}\longrightarrow {\mathbb R}$ be a Lipschitz map
with first derivative continuous in the set $X$. Assume that the
orbit ${\mathcal O}_{x_1}=\{x_1, x_2, \ldots, x_n, \ldots \}\subset
X$ satisfies $f^{'}(x_i) \neq 0$ for all $ i = 1, 2, \ldots $. If
${\mathcal O}_{x_1}$ is asymptotically periodic to the periodic
orbit ${\mathcal O}_{y_1} = \{y_1, y_2, \ldots, y_k, y_1, y_2,
\ldots , y_k, \ldots  \}$, then $h(x_1)= h(y_1)$, if both Lyapunov
exponent exist.
\end{theorem}
\begin{proof}
Since $f$ is a Lipschitz map, applying Rademacher's theorem with
$n=m=1$, it follows that $f$ has derivative in the set $X$. As
${\mathcal O}_{x_1}\subset X$, we can guarantee the derivative of
all point of ${\mathcal O}_{x_1}$. Although from here the proof is
similar to the proof of \cite[Theorem 3.4]{Alligood:1996}, we even
present it here for completeness.

Assume that $k=1$; then $\displaystyle{\lim_{n\rightarrow\infty}}
x_n = y_1$. Since the derivative is continuous, it follows that
$\displaystyle{\lim_{n\rightarrow\infty}} f^{ '}(x_n)=f^{ '}(y_1)$.
Moreover, one has $\displaystyle{\lim_{n\rightarrow\infty}}\ln
|f^{'}(x_n)| = \ln |f^{'}(y_1)|$. Therefore, $h(x_1) =
\displaystyle{\lim_{n\rightarrow\infty}} 1/n
\displaystyle{\sum_{i=1}^{n}}\ln |f^{'}(x_i)|=\ln
|f^{'}(y_1)|=h(x_1)$. If $k > 1$, we know that $y_1$ is a fixed
point of $f^{k}$ and ${\mathcal O}_{x_1}$ is asymptotically periodic
under $f^{k}$ to ${\mathcal O}_{y_1}$. Applying the reasoning above
to $x_1$ and $f^{k}$ it follows that $h(x_1)=\ln
|{(f^{k})}^{'}(y_1)|$. It is known that if $L$ is the Lyapunov
number of ${\mathcal O}_{x_1}$ under the map $f$, then the Lyapunov
number of ${\mathcal O}_{x_1}$ under the map $f^{k}$ is $L^{k}$ (
see \cite[Ex T3.1]{Alligood:1996}). Then the Lyapunov exponent
$h(x_1)$ of $x_1$ under $f$ equals $h(x_1) = 1/k \ln
|(f^{k})^{'}(y_1)|= h(y_1)$. The proof is complete.
\end{proof}

A version of Theorem~\ref{L4} to locally Lipschitz maps is given
below.

\begin{theorem}\label{L5}
Let $f: {\mathbb R}\longrightarrow {\mathbb R}$ be a locally
Lipschitz map on the open interval $(a, b)$ with first derivative
continuous in $(a, b)$. Assume that the orbit ${\mathcal
O}_{x_1}\subset (a, b)\cap X$ satisfies $f^{'}(x_i) \neq 0$ for all
$ i = 1, 2, \ldots $. If ${\mathcal O}_{x_1}$ is asymptotically
periodic to the periodic orbit ${\mathcal O}_{y_1} = \{y_1, y_2,
\ldots, y_k, y_1, y_2, \ldots , y_k, \ldots  \}$, then $h(x_1)=
h(y_1)$ if both Lyapunov exponent exist.
\end{theorem}

Let $f: {\mathbb R}\longrightarrow {\mathbb R}$ be a map, and let
${\mathcal O}_{x_1}$ be a bounded orbit of $f$. Recall that the
orbit is \emph{chaotic} if ${\mathcal O}_{x_1}$ is not
asymptotically periodic and if the Lyapunov exponent $h(x_1)$ is
greater than zero. In terms of Lipschitz maps one has the following
new definition for chaotic orbits:

\begin{definition}
Let $f: {\mathbb R}\longrightarrow {\mathbb R}$ be a Lipschitz map
with first derivative continuous at $X$, and let ${\mathcal
O}_{x_1}$ be a bounded orbit of
$f$. We call the orbit chaotic if\\
1. ${\mathcal O}_{x_1}$ is not asymptotically periodic;\\
2. the Lyapunov exponent $h(x_1)$ is greater than zero.
\end{definition}

\begin{remark}
It is interesting to note that, since Lipschitz maps are
differentiable almost everywhere (with respect to the Lebesgue
measure) according to Theorems~\ref{rad1}~and~\ref{rad2}, then the
numerical simulations are performed similarly to the standard case,
i.e., the cases where the map is differentiable. Because of this
fact, we do not present numerical simulations provement in this
paper.
\end{remark}

To finish this section, we give three examples. The first one shows
a map which is locally Lipschitz but it is not differentiable. The
second presents a family of Lipschitz maps which are not
differentiable; the third example consists in a family of logistic
maps which are both Lipschitz and differentiable.

\begin{example}\label{exa1}
Let us consider the map $f: {\mathbb R}\longrightarrow {\mathbb R}$
given by
\begin{eqnarray*}
x_{n+1} = f(x_{n})= \left \{
 \begin{array}{cc}
2x_{n}, & x_{n}<0\\
 x_{n}^{2}, & 0 \leq x_{n}<1 \\
0.5 x_{n} + 0.5, & x_{n}\geq 1\\
\end{array}
\right. ,
\end{eqnarray*}
Figure~\ref{fig1} shows the graphic of the map $f$. Note that $f$ is
not differentiable at $p=0$ and $p=1$, but it is locally Lipschitz
in the open intervals $(-\infty , 0)$, $(0 , 1)$
$(1, + \infty)$. Thus, our method can be applied.\\
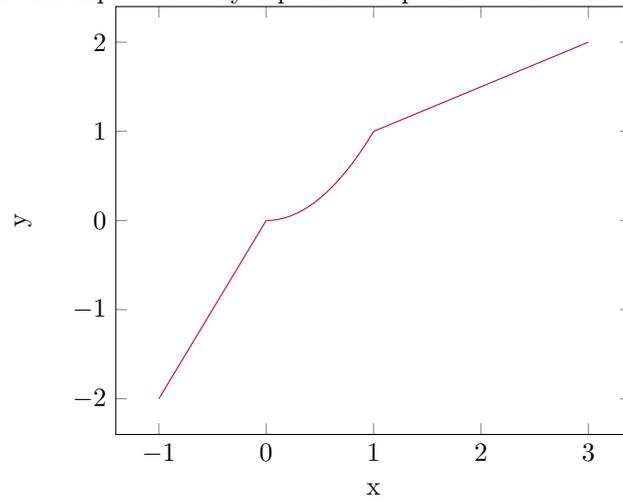
\begin{figure}
\begin{center}
\caption{Example of locally Lipschitz map which is not
differentiable\label{fig1}}
\begin{tikzpicture}
\begin{axis}[xlabel=x,ylabel=y]
\addplot[red!70!blue, samples=1200, domain=0.0:-1.0]{2*x}; \addplot[
red!70!blue,
 samples=1200, domain=0.0:1.0]{x^2};
\addplot[red!70!blue, samples=1200, domain=1.0:3.0]{0.5*x +0.5};
\end{axis}
\end{tikzpicture}
\end{center}
\end{figure}
\end{example}

\begin{example}\label{exa2}
Here, let us consider the family of Lipschitz maps $g_{a,b}:
{\mathbb R}\longrightarrow {\mathbb R}$ given by
\begin{eqnarray*}
x_{n+1} = g_{a,b}(x_{n})= a|x_{n}|+ b,
\end{eqnarray*}
where $a, b$ are real numbers. For $a=-2$ and $b=1$, the graphic of
$g_{-2,1}$ is shown in Figure~\ref{fig2}. Note that the map
$g_{-2,1}$ is Lipschitz, so our method can be applied, but it is not
differentiable. These maps are a type of tent maps.

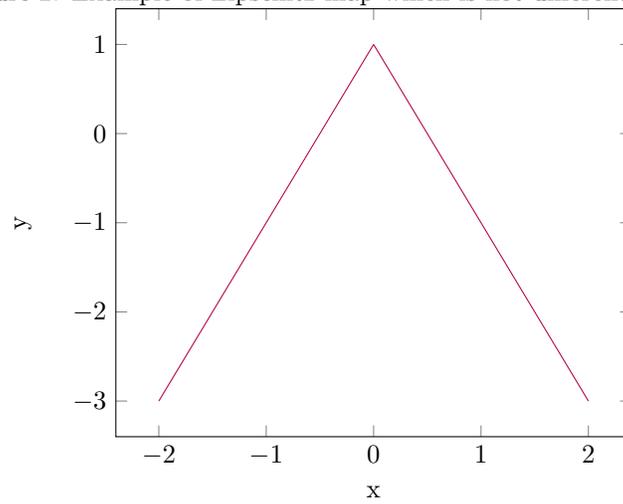
\begin{figure}
\begin{center}
\caption{Example of Lipschitz map which is not
differentiable\label{fig2}}
\begin{tikzpicture}
\begin{axis}[xlabel=x,ylabel=y]
\addplot[red!70!blue, samples=1200, domain=0.0:2.0]{-2*x+1};
\addplot[ red!70!blue,
 samples=1200, domain=-2.0:0.0]{2*x + 1};
\end{axis}
\end{tikzpicture}
\end{center}
\end{figure}
\end{example}

\begin{example}\label{exa3}
Let $f: [0, 1]\longrightarrow [0, 1]$ be a map given by
\begin{eqnarray*}
x_{n+1} = f(x_{n})= a x_{n}(1 - x_{n}),
\end{eqnarray*} where $0 < a < 4$ is a real number. These
(family of) maps are well known in the literature as logistic maps.
It is easy to see that $f$ is a Lipschitz (and also differentiable)
map. Figure~\ref{fig3} displays the graphic of the logistic map for
$a=3$.

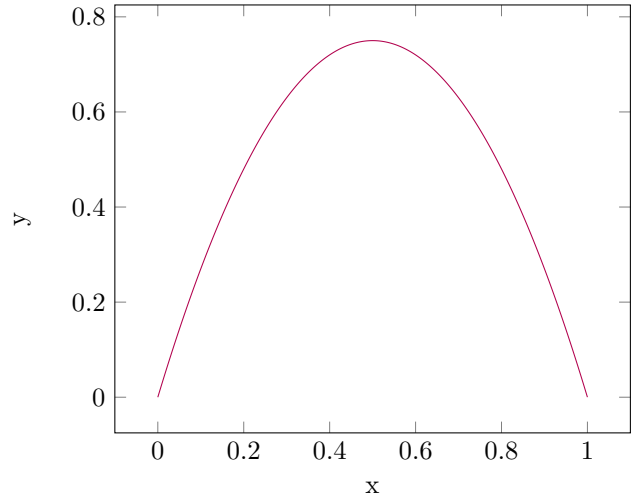
\begin{figure}
\begin{center}
\caption{Example of a Lipschitz map which is also
differentiable\label{fig3}}
\begin{tikzpicture}
\begin{axis}[xlabel=x,ylabel=y]
\addplot[red!70!blue, samples=1200, domain=0.0:1.0]{3*x*(1 - x)};
\end{axis}
\end{tikzpicture}
\end{center}
\end{figure}

\end{example}

\section{Final Remarks}\label{sec4}

We have proposed a definition of Lyapunov exponent for Lipschitz
maps, which are not necessarily differentiable. Furthermore, we have
shown that the results which are valid to discrete standard
dynamical systems are also true in this new context. This novel
approach expands the theory of dynamical systems and it is simple to
be applied. As a future work, it will be interesting to investigate
the possibility of defining Lyapunov exponents for other class of
maps, such as Holder continuous maps.

\begin{center}
\textbf{Acknowledgements}
\end{center}

This research has been partially supported by the Brazilian Agencies
CAPES and CNPq. We would like to thank the anonymous referees for
their valuable suggestions and comments that helped to improve
significantly the quality and the readability of this paper, and
Prof. Antonio Marcos Batista for helpful discussions. We also would
like to thank the Associate Editor Stefano Lenci and the
Editor-in-Chief Walter Lacarbonara for their excellent works on the
review process. The authors declare that there is not conflict of
interests in the publication of this paper.

\end{document}